\documentclass[a4paper]{amsart} 
\usepackage[utf8]{inputenc} 
\usepackage{hyperref} 
\usepackage{graphicx}
\usepackage{color}

\usepackage{amsmath,amssymb}

\theoremstyle{plain}
\newtheorem{theorem}{Theorem}[section]
\newtheorem{proposition}[theorem]{Proposition}
\newtheorem{lemma}[theorem]{Lemma}   
\newtheorem{corollary}[theorem]{Corollary}

\theoremstyle{definition}

\newtheorem{problem}[theorem]{Problem}
\newtheorem{notation}[theorem]{Notation}

\newtheorem{example}[theorem]{Example}

\newtheorem{definition}[theorem]{Definition}
\newtheorem{remark}[theorem]{Remark}

\newcommand{\enm}[1]{\ensuremath{#1}}          %

\newcommand{\cal}[1]{\mathcal{#1}}

\newcommand{\PP}{\enm{\mathbb{P}}}

\newcommand{\TT}{\enm{\mathbb{T}}}

\newcommand{\Aa}{\enm{\cal{A}}}
\newcommand{\Bb}{\enm{\cal{B}}}

\newcommand{\Ii}{\enm{\cal{I}}}

\newcommand{\Oo}{\enm{\cal{O}}}

\renewcommand{\phi}{\varphi}
\renewcommand{\theta}{\vartheta}
\renewcommand{\epsilon}{\varepsilon}

\DeclareMathOperator{\red}{red}
\DeclareMathOperator{\reg}{reg}

\renewcommand{\to}[1][]{\xrightarrow{\ #1\ }}

\newcommand{\old}[1]{}

\newcommand{\Res}{\mathrm{Res}}

\title[] 
{Terracini loci and a codimension one Alexander-Hirschowitz Theorem} 
\author{E. Ballico, M.C. Brambilla, C. Fontanari} 

\address{Edoardo Ballico, Dipartimento di Matematica, Universit\`a degli Studi di Trento, Via Sommarive 14, 38123 Povo, Trento, Italy} 
\email{edoardo.ballico@unitn.it} 

\address{Maria Chiara Brambilla, Universit\`a Politecnica delle Marche, Via Brecce Bianche, 60131 Ancona, Italy} 
\email{m.c.brambilla@univpm.it} 

\address{Claudio Fontanari, Dipartimento di Matematica, Universit\`a degli Studi di Trento, Via Sommarive 14, 38123 Povo, Trento, Italy} 
\email{claudio.fontanari@unitn.it} 



\begin{document} 

\begin{abstract} 
The Terracini locus $\mathbb{T}(n, d; x)$ is the locus 
of all finite subsets $S$ of $ \mathbb{P}^n$ of 
cardinality $x$ such that $\langle S \rangle = \mathbb{P}^n$, 
$h^0(\mathcal{I}_{2S}(d)) > 0$, and $h^1(\mathcal{I}_{2S}(d)) > 0$. 
The celebrated Alexander-Hirschowitz Theorem  classifies the triples $(n,d,x)$ for which $\dim\mathbb{T}(n, d; x)=xn$. 
Here we fully characterize the next step in the case $n=2$, namely, we prove that 
$\mathbb{T}(2,d;x)$ has at least one irreducible component of dimension $2x-1$ if and only if either $(d,x)\in\{(4,4),$$(4,6),$
$(5,6),(5,7),$ $(6,9),(6,10)\}$, or $d\ge 7$, $d\equiv 1,2 \pmod{3}$
  and $x=(d+2)(d+1)/6$. 
\end{abstract} 

\maketitle 

\section{Introduction} 

{Let $\PP^n$ be the projective space over an algebraically closed field of characteristic zero.
The celebrated Alexander-Hirschowitz Theorem (see e.g. \cite{AH, BO, pos, AM}) classifies all linear systems of hypersurfaces of $\PP^n$ which are singular at a given number of general points and do not have the expected dimension. More explicitly:
\begin{theorem}[Alexander-Hirschowitz]\label{AHthm}
Given $s$ points of $\PP^n$ in general position, the linear system of degree $d$ hypersurfaces of $\PP^n$ which are singular at these points has not the expected dimension $$\max\left\{\binom{n+d}n-s(n+1)-1,-1\right\}$$ 
if and only if 
\begin{itemize}
\item
 either $(n,d,x) = (n,2,x)$ with $2 \le x \le n$, 
 \item or $(n,d,x) \in\{ (2,4,5), (3,4,9), (4,3,7),(4,4,14)\}$. 
 \end{itemize}
\end{theorem}
}

{
Linear systems whose dimension is greater than the expected one are called \emph{special}. The speciality of  linear systems of $\PP^n$ is related to the defectivity of higher secant varieties of the Veronese varieties, see for instance \cite{BCCGO} for more details, {and, consequently, is very important for many applications.}}

{A new object of study in this setting is the Terracini locus of a projective variety, see \cite{ballico-chiantini}. 
Since then, it has been
investigated by many authors, see in particular \cite{bbs}, \cite{ GSTT} and the references therein.
Roughly speaking, the Terracini locus parametrizes the set of all points of a projective variety such that the linear systems of hypersurfaces singular at them are special.} 

{In this paper we focus on the case of Veronese embeddings of projective spaces.
}
For any positive integers $x$, 
let $S(\PP^n, x)$ denote the set of all $A \subset \PP^n$ of cardinality $x$, endowed with the Zariski topology. 
Then the \emph{Terracini locus} $\mathbb{T}(n, d; x)$ is the set of all $S \in S(\mathbb{P}^n, x)$ 
such that $$\langle S \rangle = \mathbb{P}^n,\quad
h^0(\mathcal{I}_{2S}(d)) > 0,\quad h^1(\mathcal{I}_{2S}(d)) > 0.$$

Assume $n \ge 2$. By \cite{bb1}, Theorem 1.1, we recall that the Terracini locus 
$\mathbb{T}(n, d; x)$ is empty if $d=2$ and if $(n,d) = (2,3)$. 
On the other hand, if $d \ge 3$ and $(n, d) \ne (2, 3)$, then $\mathbb{T}(n, d; x) \ne \emptyset$ if and only if 
$x \ge n + \lceil d/2 \rceil$. 

{By an obvious parameter count we have $\dim\mathbb{T}(n, d; x) \le xn$.
The Alexander-Hirschowitz Theorem \ref{AHthm} can be rephrased as follows:
\begin{theorem}[Alexander-Hirschowitz]\label{AHthm2}
We have $\dim\mathbb{T}(n, d; x)=xn$ if and only if
\begin{itemize}
\item
 either $(n,d,x) = (n,2,x)$ with $2 \le x \le n$, 
 \item or $(n,d,x) \in\{(2,4,5), (3,4,9), (4,3,7),(4,4,14)\}$. 
 \end{itemize}
\end{theorem}
}

It seems natural to try and find out which are the triples  of integers 
$(n,d,x)$ for which we have at least one irreducible component 
$V \subseteq \mathbb{T}(n, d; x)$ of fixed dimension $\dim(V)<nx$.

Already in the first case of dimension $xn - 1$ the question above turns out to be widely open. More precisely, we pose the following:

\begin{problem}
Let $n \ge 2$, $d \ge 3$, $(n, d) \ne (2, 3)$, and $x \ge n + \lceil d/2 \rceil$, so that 
$\mathbb{T}(n, d; x) \ne \emptyset$. 
Determine all triples $(n,d,x)$ such that $\mathbb{T}(n,d;x)$ has at least one irreducible component of dimension $nx-1$.
\end{problem}

Indeed, this task seems to be nontrivial even in the case $n=2$. Our main result is the following:

\begin{theorem}\label{finale2} 
Fix integers $x>0$ and $d\ge 4$. The locus $\TT(2,d;x)$ has at least one irreducible component of dimension $2x-1$ if and only if 
\begin{itemize}
\item
either $(d,x)\in\{(4,4),(4,6),(5,6),(5,7),(6,9),(6,10)\}$
\item or  $d\ge 7$, $d\equiv 1,2 \pmod{3}$, and $x=(d+2)(d+1)/6$.
\end{itemize}
Moreover, for $d\ge 7$, such a component is unique with the only exception of $(d,x)=(8,15)$, where there are exactly two components.
\end{theorem} 

{Note that, for $n=2$, in all the cases which are not listed in Theorems \ref{AHthm2} and \ref{finale2}, we have that $\dim(\TT(2,d;x))\le 2x-2$.}

{
For many reasons (see for instance \cite{AC} and \cite{CFM}),
divisors on parameter spaces are very important. Anyway, they should be defined on a projective parameters space, not on a Zariski open subset of it. Indeed divisors should have an intrinsic definition, it is not sufficient to say: \emph{Take the Zariski closure of the divisor D of $U$}. In the set-up of Terracini loci, the natural compactification of $S(X_{\reg},x)$ is the smoothable component of the Hilbert scheme of $x$ points of $X$. If $X$ is a smooth surface, then this is the full Hilbert scheme of $x$ points of $X$
(see \cite{f}).
}

{Finally, we would like to point out that  one of the main tools we used in order to prove our result is the {\em spread} $\eta$.} 
More precisely, for any locally closed irreducible set $K\subseteq S(\PP^n,x)$ let $\eta(K)$ be the maximal integer $y$ such that for a  general  $S'\in S(\PP^n,y)$ there exists $S\in K$ containing $S'$ (see Notation \ref{eta} for the precise definition). 
We think that it is interesting and useful to study the spread $\eta(K)$ for the irreducible components $K$ of $\TT(n,d;x)$.
In  Section \ref{sec eta} we  start this investigation, while collecting the ingredients for the proof of Theorem \ref{finale2}.

We work over an algebraically closed field $\mathbb{K}$ of characteristic $0$. 
\medskip

{We would like to thank the referees for several useful comments that helped us to significantly improve our manuscript.}

\subsection*{Acknowledgements}
All authors are members of GNSAGA of INdAM.

E.~Ballico has been partially funded by
by the European Union Next Generation EU, M4C1, CUP B53D23009560006, PRIN 2022- 2022ZRRL4C - Multilinear Algebraic Geometry.
M.~C.~Brambilla has been partially funded by the 
European Union Next Generation EU, M4C1, CUP E53C24002320006 - 
2022NBN7TL - Applied Algebraic Geometry of Tensors.
Views and opinions expressed are however those of the authors only and do not necessarily reflect those of the European Union or European Commission. Neither the European Union nor the granting authority can be held responsible for them.

\section{Preliminaries}
If $X$ is a reducible projective variety, then we denote by $\dim(X)$ the dimension of a maximal irreducible component of $X$.

{Given a scheme $Z\subset \PP^n$, we denote by $|\Ii_Z(d)|:=\PP(H^0(\PP^n,\Ii_Z\otimes\Oo(d)))$ the linear system of hypersurfaces of degree $d$ containing $Z$.}

\begin{remark}\label{0u1}
Take $S\in S(\PP^n,x)$, $n\ge 2$, and assume $(n+1)x\ge \binom{n+d}{n}$. 
{Taking the cohomology of the exact sequence
$$0\to \Ii_{2S}(d)\to \Oo_{\PP^n}(d) \to \Oo_{2S}(d)\to 0$$
we 
} have that
$S\in \TT(n,d;x)$ if and only if $\langle S\rangle =\PP^n$ and $h^0(\Ii_{2S}(d))>0$. 
\end{remark}

\begin{remark}\label{000u3}
Fix positive integers $c$ and $x$ such that $(n+1)(x+1)\le \binom{n+d}{n}$. Assume the existence of an irreducible family $K\subseteq \TT(n,d;x)$ such that $\dim K =nx-c$.
The set of all $S\cup \{p\}$, $S\in K$, $p\in \PP^n\setminus S$, is an irreducible family $F\subseteq \TT(n,d;x+1)$ of dimension $n(x+1)-c$. 
\end{remark}

\begin{proposition} \label{prop sigma}
Fix integers $n\ge 2$ and $d\ge 3$ such that $(n,d)\notin \{(2,4),(4,3),$ $(4,4)\}$. Set \begin{equation}\label{sigma}\sigma:= \left\lfloor \frac{1}{n+1}\binom{n+d}{n}\right\rfloor.\end{equation}
Fix an integer $y<\sigma$ and assume $\dim \TT(n,d;y)\ge ny-1$. Then $\dim \TT(n,d;x) =nx-1$ for all $y\le x\le \sigma$.
\end{proposition}

\begin{proof} 
By the Alexander-Hirschowitz Theorem \ref{AHthm}, $\dim \TT(n,d;x)<nx$ for all $x\le \sigma$. Thus it is sufficient to prove that $\dim \TT(n,d;x)\ge nx-1$. 
{We prove it by induction on $x\ge y$.
If $x=y$ we know that $\dim \TT(n,d;y)\ge ny-1$
by hypothesis. 
Assume that $\dim \TT(n,d;x-1)\ge n(x-1)-1$, then by Remark \ref{000u3}, with $c=1$, we conclude.}
\end{proof}

\begin{notation}\label{not}
For any positive integer $x$, let $\PP^n[x]$ denote the set of all $(p_1,\dots ,p_x)\in (\PP^n)^x$ such that $p_i\ne p_j$ for all $i\ne j$.
Let $$u_x: \PP^n[x] \to S(\PP^n,x)$$ denote the map $(p_1,\dots ,p_x)\mapsto \{p_1,\dots ,p_x\}$. 
For every $y$ with $1\le y\le x$, let
$$\eta_y: \PP^n[x]\to  \PP^n[y]$$
denote the projection onto the first $y$ factors of $(\PP^n)^x$ .
\end{notation}
Observe that $u_x$ is a finite and unramified map with fibers of cardinality $x!$. Thus for any locally closed irreducible set $K\subseteq S(\PP^n,x)$ the
set $u_x^{-1}(K)$ has pure dimension $\dim (K)$. 

\begin{notation}\label{notation rho}
Given integers $n\ge 2$ and  $d\ge 3$ we set $$\rho(n,d):= \left\lceil \frac{1}{n+1}\binom{n+d}{n}\right\rceil.$$ 
\end{notation}

The following proposition gives a first easy bound on the dimension of the Terracini locus when the number of points is high.
\begin{proposition}\label{super-abundant}
Fix integers $n\ge 2$, $d\ge 3$ and $c>0$ such that $(n,d)\ne (2,3)$. 
Set $\zeta=\rho(n,d)$ 
if $(n,d)\notin \{(2,4),(3,4),(4,3),(4,4)\}$
and $\zeta=\rho(n,d) +1$ otherwise.
Set $x_0:= c \zeta$ and fix an integer $x\ge x_0$. Then we have $\dim \TT(n,d;x)\le nx-c$. 
\end{proposition}

\begin{proof} 
Let $K$ be an irreducible component of $\TT(n,d;x)$. By the Alexander-Hirschowitz Theorem \ref{AHthm}, we have $h^0(\Ii_{2A}(d)) =0$ for a general 
$A\in S(\PP^n,\zeta)$. Hence 
$\dim \eta_{\zeta}(u_x^{-1}(K)) \le n\zeta -1$.
Fix a general $S\in K$ and label the $x$ points of $S$ as $p_{i,j}$, $1\le i\le \zeta$, $1\le j{\le c}$, 
and call $q_\alpha$ the other points (if $x>x_0$).
Varying $S$ in $K$ for each fixed $j$ each set $\{p_{i,j}, 1\le i\le \zeta\}$ depends on at most $n\zeta-1$ parameters. The set of all $q_\alpha$'s depends on at most
$n(x-x_0)$ parameters. Thus $\dim K\le {c(n\zeta-1)+n(x-x_0)}= nx-c$.
\end{proof}

\section{The general case}

\label{sec eta}
\begin{definition}\label{eta} 
For any locally closed irreducible {non-empty} set $K\subseteq S(\PP^n,x)$, let $\eta(K)$ be the maximal integer $y$ such that $\eta_y(u_x^{-1}(K))$ contains a non-empty open subset of $\PP^n[y]$.
\end{definition}

\begin{remark}\label{rmk1:eta} 
The integer $\eta(K)$, {which we will call {\it spread of $K$},}
is the {maximal} integer $y$ such that for general $S_1\in S(\PP^n,y)$
there is $S\in K$ containing $S_1$. {Hence}
for any general $S_2\in S(\PP^n,y+1)$  there is no $S\in K$ containing $S_2$. Equivalently $\eta(K)$ is the maximal integer such that the map $\eta_y$ is dominant on 
$u_x^{-1}(K)$.
\end{remark}

\begin{remark}\label{rmk:eta}
Take an irreducible family $K\subseteq \TT(n,d;x)$ and set $\eta:=\eta(K)$.  
By definition of $\eta(K)$ we have 
\begin{equation}\label{eqalpha1} 
n\eta \le \dim K \le  n\eta +(n-1)(x-\eta){=(n-1)x+\eta}.
\end{equation} 
Indeed, by Remark \ref{rmk1:eta}, we can choose $\eta$ general points in $K$, but we cannot generically choose  the other $x-\eta$ points.

Let $K$ be an irreducible component of $\TT(n,d;x)$.
If $y=\dim K-(n-1)x>0$,
then by \eqref{eqalpha1} we have $\eta(K) \ge y$. 

Therefore, if $\dim(K)=nx-1$, we have that the spread $\eta(K)=x-1$ . 
\end{remark}

\begin{theorem}\label{alpha2} 
Take $n\ge 2$ and $d\ge 3$, such that  $(n,d)\notin \{(2,4),(3,4)$,$(4,3),(4,4)\}$ and
set 
$x\ge \rho(n,d)$, as in Notation \ref{notation rho}.
 Then $\TT(n,d;x)$ has no irreducible component $K$ with $\eta(K)\ge x$. 
\end{theorem} 

\begin{proof} 
By \eqref{eqalpha1} we have
$$\dim(\TT(n,d;x))\ge\dim(K)\ge n\eta(K).$$

On the other hand, by Proposition \ref{super-abundant}, with $c=1$, we have
$$\dim(\TT(n,d;x))\le nx-1.$$

Hence we conclude that {the spread satisfies} $\eta(K)\le x-1.$
\end{proof} 

\begin{proposition}\label{bo1}
Take $(n,d,x)\in \{(2,4,6),(3,4,10),(4,4,15)\}$. Then 
there is a unique irreducible component of $\TT(n,d;x)$ of dimension $nx-1$.
\end{proposition}

\begin{proof}
Since $x>\rho(n,d)$, by the Alexander-Hirschowitz Theorem \ref{AHthm} we have that $h^0(\Ii_{2A}(d)) =0$ for a general $A\in S(\PP^n,x)$. 
Hence $\dim \TT(n,d;x) \le nx-1$.
Take a general $B\in S(\PP^n,x-1)$. Then we have $|\Ii_{2B}(d)|=\{2T_B\}$, where $T_B$ is the unique smooth quadric containing  $B$. Hence, varying
$B$, the set of all $B\cup \{p\}$ with $p\in T_B\setminus B$ forms an irreducible family $K$ of $\TT(n,d;x)$ of dimension $nx-1$. Obviously, $\overline{K}$ is the unique irreducible family of $\TT(n,d;x)$ with spread $x-1$. Hence $\overline{K}$ is the  unique irreducible component of $\TT(n,d;x)$ of dimension $nx-1$.
\end{proof}

\begin{remark}\label{rem-deg-4}
Adapting the proof of the previous proposition, it is easy to prove that if $(n,d,x)\in \{(2,4,6),(3,4,10),(4,4,15)\}$, we have $\dim \TT(n,d;y) \le ny-2$ for any $y\ge x+1$.
\end{remark}

\begin{proposition}\label{bo2}
Take $(n,d,x)=(4,3,8)$. Then $\dim \TT(n,d;x) \le nx-2$.
\end{proposition}

\begin{proof}
As above since $8=x>\rho(4,3)=7$, we have $h^0(\Ii_{2A}(3)) =0$ for a general $A\in S(\PP^4,8)$,
hence $\dim \TT(n,d,x) \le nx-1$.
Take a general $B\in S(\PP^4,7)$. Let $C_B\subset \PP^4$ be the unique rational normal curve containing $B$.
By the Alexander-Hirschowitz Theorem \ref{AHthm} and \cite{BO}, \S 3, we have 
$|\Ii_{2B}(d)|=\{\sigma_2(C_B)\}$, where $\sigma_2(C_B)$ is the secant variety to the curve $C_B$.  The cubic hypersurface $\sigma_2(C_B)$ has two orbits for the action of $\mathrm{Aut}(\PP^1)$ and the singular locus is $\mathrm{Sing}(\sigma_2(C_B))=C_B$ (see the determinantal equation in \cite{BO}, \S 3). 
Hence the irreducible family $K$ of maximal dimension contained in $\TT(4,d;x)$ is given by the set of all $B\cup \{p\}$ with $p\in C_B\setminus B$, varying $B$, and clearly $\dim(\overline{K})\le nx-2$.
\end{proof}

\begin{theorem}\label{000u2} 
Given integers $n\ge 2$ and $d\ge 3$,
set $\rho=\rho(n,d)$ as in Notation \ref{notation rho} and $x\ge \rho+1$.
Assume $(n,d,x)\notin \{(2,4,6),(3,4,10),(4,4,15)\}$.
Then we have $\dim \TT(n,d;x)\le nx-2$.
\end{theorem}
\begin{proof}  
Assume the existence of an irreducible component $K$ of $\TT(n,d;x)$ such that $\dim K\ge nx-1$.
Hence $\eta(K) \ge x-1$ {by Remark \ref{rmk:eta}}. Let $K_1\subseteq S({\PP^n},x-1)$ be the subset of $S({\PP^n},x-1)$ formed by all subsets of cardinality
$x-1$ contained in an element of $K$. Since $\eta(K) \ge x-1$, there is an irreducible component $K_2$ of $K_1$ with $\eta(K_2)=x-1$, i.e. $K_1$ contains an open subset of $S({\PP^n},x-1)$.

If $(n,d,x)=(4,3,8)$ then 
 the conclusion follows from Proposition \ref{bo2}.
Otherwise,
since $(n,d,x-1)$ is not an exception of the Alexander-Hirschowitz Theorem \ref{AHthm}, we have
that $h^0(\Ii_{2S_1}(d)) = 0$ for a general $S_1\in K_2$. 
By construction there exists $S\in K$ such that 
$S\supset S_1$. Hence we have $h^0(\Ii_{2S}(d)) =0$, contradicting the definition of Terracini set.  
\end{proof}

The next result uses our assumption that each element of $\TT(n,d;x)$ spans $\PP^n$. 

\begin{lemma} 
Let $K$ be an irreducible component of $\TT(n,d;x)$. Then $\eta(K)\ge n+1$.
{Moreover,
$\eta(K)\ge n+2$ if and only if there is $S\in K$ containing $n+2$ points in linear general position.}
\end{lemma}

\begin{proof}
We have $h(\TT(n,d;x))=\TT(n,d;x)$ for all $h\in \mathrm{Aut}(\PP^n)$. Since $\mathrm{Aut}(\PP^n)$ is irreducible, $h(K)=K$ for all $h\in \mathrm{Aut}(\PP^n)$.
Since each $S\in K$ spans $\PP^n$, then $S$ contains $S'\in S(\PP^n,n+1)$ such that $\langle S'\rangle =\PP^n$. Since all {elements of} $\mathrm{Aut}(\PP^n)$ act transitively
on the open subset of $S(\PP^n,n+1)$ formed by linearly independent points, we conclude that $\eta(K)\ge n+1$.

Recall that $n+2$ general points of $\PP^n$ are in linear general position, {if} any {subset of} $n+1$ of these points spans $\PP^n$. The group $\mathrm{Aut}(\PP^n)$ acts transitively on the subset of $S(\PP^n,n+2)$ formed by points in linear general position. Thus {arguing as above,} we can prove  that $\eta(K)\ge n+2$ if and only if there is $S\in K$ containing $S''\in S(\PP^n,n+2)$ in linear general position.
\end{proof}

\begin{theorem}
\label{THM 39}
Take $n\ge 2$, $d\ge 3$ and $x$ such that 
$x\ge n+\lceil d/2\rceil$ and $(n+1)x\le \binom{n+d}{n}$. Then there is an irreducible family $K$
of $\TT(n,d;x)$ such that $\eta(K) = x-\lceil d/2\rceil+1$ and 
{$\dim K = \lceil d/2\rceil +n-1+n(x-\lceil d/2\rceil)$.}
\end{theorem}

\begin{proof}
Fix any $S\in S(\PP^n,x)$. Since $\deg(2S) =(n+1)x\le \binom{n+d}{n}$, we have that $S\in \TT(n,d;x)$ if and only if $\langle S\rangle =\PP^n$ and $h^1(\Ii_{2S}(d)) >0$.

We will first define an irreducible quasi-projective variety $K_1\subseteq \TT(n,d;x)$ 
and then we will take as $K$ the closure of $K_1$ in $\TT(n,d;x)$.

Let $G(2,n+1)$ be the Grassmannian of lines in $\PP^n$. 
{Consider the set $$E\subset S(\PP^n,\lceil d/2\rceil +1)\times G(2,n+1)$$ formed by all pairs $(A,L)$ with $A\in S(\PP^n,\lceil d/2\rceil +1)$, $L\in G(2,n+1)$ such that $A\subset L$.}
Since $\dim G(2,n+1)=2n-2$, {the projection $$S(\PP^n,\lceil d/2\rceil +1)\times G(2,n+1)\to G(2,n+1)$$ shows that $E$ is an irreducible quasi-projective variety of dimension $(2n-2)+ (\lceil d/2\rceil+1).$ 
Let $K_2\subset S(\PP^n,\lceil d/2\rceil +1)$
be the image of $E$ by the projection to the first factor. Since $\lceil d/2\rceil+1> 2$ and any two points are contained in a unique line, $K_2$ is an irreducible constructible subset of $S(\PP^n,\lceil d/2\rceil +1)$ and $\dim K_2= 2n-2+ \lceil d/2\rceil +1$. The set $K_2$ is the the set of all collinear elements of $S(\PP^n,\lceil d/2\rceil +1)$. Now, let $K_1$ be the set of all $S\in S(\PP^n,x)$ containing an element of $K_2$ and such that $\langle S\rangle=\PP^n$.}
Notice that, since $\deg(2S\cap L)\ge 2(\lceil d/2\rceil +1)\ge d+2$, then we have $h^1(\Ii_{2S}(d))\ge h^1(\Ii_{2S\cap L}(d)) >0$. Therefore $K_1\subseteq \TT(n,d;x)$.

{Let $K$ be the closure of $K_1$.}
We have $\dim(K)=(2n-2+\lceil d/2\rceil +1) +n(x-1-\lceil d/2\rceil)$.  
It is easy to check that $\eta(K) = (x-\lceil d/2\rceil-1)+2=x-\lceil d/2\rceil+1$. 
\end{proof}

\begin{example}\label{ex1} 
Take $n=2$, $d=4$ and $x=4$.
The hypothesis of Theorem \ref{THM 39} are verified, hence there is a component of $\TT(2,4;4)$ of dimension $7$.
Note that we have $\TT(2,4;4)\ne S(\PP^2,4)$, because $h^1(\Ii_{2A}(4)) =0$ if $A$ is given by $4$ points in linear general position.
\end{example}

The next example shows that $\TT(5,4;21)$ contains a codimension one variety whose general member is \emph{minimally Terracini} in the sense of \cite{bb1}.
\begin{example}\label{n5d4}
Take $n=5$, $d=4$ and $x=21$. We have $h^0(\Oo_{\PP^5}(4))/(n+1) = \binom{9}{4}/6 =21$. By the Alexander-Hirschowitz Theorem \ref{AHthm}, we have $h^i(\Ii_{2S}(4)) =0$, $i=0,1$, for a general
$S\in S(\PP^5,x)$. We have $h^0(\Oo_{\PP^5}(2)) =21$. Thus $h^0(\Ii_{S'}(2)) =1$ for a general $S'\in S(\PP^5,20)$. Consider the $104$-dimensional irreducible family $\Psi\subset S(\PP^5,21)$
given by all $S$ such that $h^0(\Ii_S(2)) =1$ and such that the only element of $|\Ii_S(2)|$ is irreducible. 
{Since $h^1(\Ii_{2S}(4))= h^0(\Ii_{2S}(4))>0$, then $S\in\TT(5,4;21)$.
Moreover, we prove that $S$ is minimally Terracini.}
Fix any irreducible $Q\in |\Oo_{\PP^5}(2)|$ and take a general $S\in S(Q,21)$.
By the generality of $S$ we have $|\Ii_{S_1}(2))|=\{Q\}$ for all $S_1\subset S$ such that $\#(S_1)=20$. Hence $h^1(\Ii_{S_1}(2))=0$. The residual exact sequence of $Q$ gives $h^1(\Ii_{2S_1}(4)) =0$. 
\end{example}

The following result gives a complete description of $\TT(3,3;5)$.

\begin{proposition}\label{t17.1}
$\TT(3,3;5)$ is irreducible of dimension $14$, formed by all $S\in S(\PP^3,5)$ such that $\langle S\rangle =\PP^3$ and $4$ of the points of $S$ are coplanar. The action of $PGL(4)$ on $S(\PP^3,5)$ sends $\TT(3,3;5)$ into itself with two orbits: an orbit $\Phi$ of dimension $13$ formed by all $S\in \TT(3,3;5)$ containing $3$ collinear points and the open orbit
$\TT(3,3;5)\setminus \Phi$.\end{proposition}

\begin{proof} 
The group $PGL(4)$ acts on $S(\PP^3,4)$ and on $S(\PP^3,5)$ and in both cases it has an open orbit in the Zariski topology. The open orbit of $S(\PP^3,4)$ is formed by the linearly independent subsets. The open orbit of $S(\PP^3,5)$ is formed by all $S\in S(\PP^3,5)$ in linearly general position, i.e. the set $S$ such that all proper subsets of $S$ are linearly independent.

{The action of $PGL(4)$ on the subset of $S(\PP^3, 5)$ given by the sets which span $\PP^3$ has three orbits. The open orbit is given by points in linearly general position.
A second orbit $\Psi$ (of dimension 14) 
is formed by the sets $S$ such that no three of its points are collinear, but there is $A\subset S$ such that $\#A =4$ and $A$ spans a plane.
The third orbit $\Phi$ (of dimension 13) is formed by the sets  containing  three collinear points.} Clearly $\Phi\subset \overline\Psi$.

We prove now that $ \TT(3,3;5)=\overline\Psi$.
Clearly, by the Alexander-Hirschowitz Theorem \ref{AHthm} the elements of the open orbit are not Terracini. Hence $ \TT(3,3;5)\subseteq \overline\Psi$. 

Now we prove the other inclusion. Let $S\in \Psi$, that is {assume that} $\#(S)=5$, $\langle S\rangle =\PP^3$ 
{and there exists $A\subset S$ such that $\#A=4$ and $\dim \langle A\rangle =2$.
}

Let $S'\subset S$ a subset of four points and  $H=\langle S'\rangle$ the plane spanned by $S'$.
Since $h^0(\Oo_{\PP^2}(3)) =10$ and $\deg(2S')=12$, then 
we have $h^1(\Ii_{2S\cap H}(3)) >0$. Hence, by \cite[Lemma 2.7]{bb1}, it follows that $h^1(\Ii_{2S}(3)) >0$.
Since $h^0(\Ii_{2S}(3))=h^1(\Ii_{2S}(3))$, we get that $\Psi\in\TT(3,3;5)$.
\end{proof}

\section{The planar case}

From now on, we fix $n=2$.

\begin{remark}\label{casipiccoli}
We know from \cite[Example 4.6 and Example 5.1]{ballico-chiantini} that if $(d,x)=(5,6)$ and $(d,x)=(6,9)$ the Terracini locus has a component of codimension $1$. {Moreover, from \cite[Example 4.6]{ballico-chiantini} it follows that the codimension of $\TT(2,5;5)$ is greater than 1.}
\end{remark}

\begin{proposition}\label{caso 6,8}
There is no irreducible component of $\TT(2,6;8)$ of codimension one.
\end{proposition}
\begin{proof}
Assume by contradiction the existence of an irreducible component $K$ of $\TT(2,6;8)$ of dimension $15=nx-1$. Then by Remark \ref{rmk:eta}, we have that the spread is $\eta(K) =7$. 
Fix a general $S'\in S(\PP^2,7)$ and a set $S\in K$ such that $S\supset S'$ and  $h^1(\Ii_{2S}(d))>0$.

By the curvilinear lemma, see e.g. \cite[Lemma 2.9]{bb1},
there exists a scheme $\xi \subset 2S$ such that any connected component of $\xi$ has
degree less or equal than $2$ and such that $h^1(\Ii_\xi(d))>0$.
Let $v$ be the connected component of $\xi$ supported at $p=S\setminus S'$. Clearly $h^1(\Ii_{2S'\cup v}(d))\ge h^1(\Ii_\xi(d))>0$.

Fix a general line $L\subset \PP^2$ containing $v$, i.e. spanned by $v$.
Consider the following residual exact sequence with respect to $L$:
\begin{equation}\label{equltima}
0 \to \Ii_{\Res_L(v\cup 2S')}(5)\to \Ii_{2S'\cup v}(6)\to \Ii_{v\cup (L\cap 2S'),L}(6)
\to 0
\end{equation}
where $\Res_L(v\cup 2S')$ is the residual with respect to $L$ and $ v\cup (L\cap 2S')$ is the trace.
Now,
since $\Res_L(v\cup 2S')=\Res_L(2S') \subseteq 2S'$, we have 
$h^1(\Ii_{\Res_H(v\cup 2S')}(5))\le h^1(\Ii_{2S'}(5)) =0$. 

Since $S'$ is general, $\#(S'\cap L)\le 2$. Hence we have $h^1(L,\Ii_{v\cup (L\cap 2S'),L}(6))=0$.

By the long cohomology exact sequence of \eqref{equltima}, we conclude $h^1(\Ii_{2S'\cup v}(d)) =0$, which is a contradiction.
\end{proof}

{\begin{notation} \label{severi}
Let $V_{x,d}$ denote the Severi variety of all irreducible degree $d$ plane curves with exactly $x$ nodes as singularities. 
{For an introduction to Severi varieties see e.g.\ \cite[pp.\ 29-32]{harris-morrison} and \cite{fedorchuck}.}
{Clearly, $V_{x,d}$ is non-empty if and only if $x \le (d-1)(d-2)/2.$}

It is known that $V_{x,d}$ is irreducible {(\cite{harris} and \cite[Chapter 6, Section E]{harris-morrison})} of 
dimension $\binom{d+2}{2}-1-x$.
Let $\phi: V_{x,d}\to S(\PP^2,x)$ denote the map $C\mapsto \mathrm{Sing}(C)$.
\end{notation}}

We recall the following theorem of Treger.
\begin{theorem}[\cite{treger}]\label{u1}
Let  $d\ge 6$, $0\le x\le (d-1)(d-2)/2$, and $V_{x,d}$ the Severi variety of the irreducible degree $d$ plane curves with exactly $x$ nodes as singularities. 
 Assume
$x\ge d(d+3)/6$
and $(d,x)\ne (6,9)$. Then the map $\phi$ defined in Notation \ref{severi} is birational onto its image.
\end{theorem}
As a consequence, we have the following result.

\begin{corollary}
Assume  $d\ge 6$ and
$${(d+2)(d+1)/6 }\le x \le (d-1)(d-2)/2.$$
Then $\TT(2,d;x)$ contains an irreducible family  of dimension $\binom{d+2}{2}-1-x$.
\end{corollary}
\begin{proof}{
Since $x\ge(d + 2)(d + 1)/6 =d(d+3)/6 +1$,
we can apply Theorem \ref{u1} and Remark \ref{0u1}. Note that for any $C\in V_{x,d}$, we have $\langle\varphi(C)\rangle=\PP^2$.
Indeed if the points $\varphi(C)$ were collinear, then the line through them would be a component of $C$, since $x>d$, and this is a contradiction, because $C$ is irreducible.
}
\end{proof}

In particular we have the following {family of examples.}
\begin{example}\label{e1}
Fix an integer $d\ge 7$ such that $d\equiv 1,2\pmod{3}$. Set $x:= (d+2)(d+1)/6$. 
Then the irreducible component $\overline{\phi(V_{x,d})}$ has dimension $2x-1$, 
because $\binom{d+2}{2}-1-x =2x-1$.
\end{example}

\begin{example}\label{caso(5,7)} 
Take $n=2, d=5, x=7$. This case is discussed in \cite[Example 5.3]{ballico-chiantini}. We give here a more detailed description
of $\TT(2,5;7)$ as a union of finitely many locally closed irreducible families, showing that there exists exactly one component $T$ of dimension $13$.

{Fix $S\in \TT(2,5;7)$. Recall that an irreducible plane quintic has at most $6$ singular points, hence any curve  $C\in|\Ii_{2S}(5)|$ is reducible and/or with multiple components.}

{By a case-by-case analysis, it is easy to see that any irreducible family of not too small dimension of $\TT(2,5;7)$ have  general member described in the following way:}

\smallskip

\quad $\bullet$
$T:=\{\mathrm{Sing}(C_3\cup C): C_3\in V_{1,3}, C \textrm{ a smooth conic}\}.$
  A general $S\in T$ is given by a point $p$ in general position and six other points obtained as the intersection of an irreducible conic with a cubic singular at $p$.
{Now we show that $\dim(T)=13$. Indeed, fix a smooth conic $C$ and take six points $S'\subset C$. The set of all such pairs $(C,S')$ has dimension $\dim |\Oo_{\PP^2}(2)|+6=5+6=11$. Note that the set of points $S'$ uniquely determines $C$, so also the set of all such $S'$ has dimension $11$.
Now, choose $p\in \PP^2\setminus C$ and take $S:= S'\cup \{p\}$. 
The set of all such $S$ has dimension $11+2=13$ and it coincides with $T\subset\TT(2,5;7)$. 
Indeed, given $S$ there is a unique cubic singular at $p$ and passing through $S'$, because
 $h^0(\Ii_{2p\cup S'}(3)) =1$.}
{We point out that, by construction, the family $T$ is irreducible.}



\smallskip

\quad $\bullet$ 
$Y:=$ the family of the sets given by five points in general position and two points on the conic through the first five, and take the closure $\overline{Y}$. The general curves in $|\Ii_{2S}(5)|$  are the unions of the double conic and a line. We have $\dim Y =12$, because $\dim |\Oo_{\PP^2}(2)|=5$ and then we add $7$ points on any fixed conic.

\smallskip

\quad $\bullet$ 
If we assume that $z\ge 4$ points of $S$ are collinear, then we are always in the irreducible family $\overline{Z}$,
where $Z$ is the family of all the sets $S$ of four collinear points and three  points in general position (not collinear with the first ones).
In particular, $Z=\{\mathrm{Sing}(C_4\cup L): C_4\in V_{3,4}, L \textrm{ a line}\}$
does not give a component of dimension higher than $12$.

\smallskip

\quad $\bullet$ 
 $U:=\{\mathrm{Sing}(C\cup L\cup N): C \textrm{ a smooth cubic}, L,N \textrm{ two lines}\}.$ In this case we have $\dim(U)=10$ for the following reasons.  We choose three general points $P_1,P_2,P_3$ and set $L=\langle P_1,P_2\rangle$ and $N=\langle P_1,P_3\rangle$, then we choose other two points on $L$ and other two points on $N$.
 Any set of $6$ points of $\PP^2$  is contained in a plane cubic.
 Hence we have $\dim(U)=\dim S(\PP^2,3)+ 2+2= 10$.
\end{example}

\begin{lemma} \label{last} Assume $d\equiv 0\pmod{3}$, 
$d\ge 9$ and
$$x = \rho(2,d)=\left\lceil \frac{1}{3}\binom{d+2}{2}\right\rceil=\frac{d(d+3)}{6}+1,$$
Then we have $\dim \TT(2,d;x)\le 2x-2$.
\end{lemma}

\begin{proof} 
Assume by contradiction the existence of an irreducible component 
$K$ of $\TT(2,d;x)$ of dimension $2x-1$. {By Remark \ref{rmk:eta} we have $\eta(K)=x-1$.}
{Hence, given} a general $S\in K$, {there is a general $S'\subset S$}
 in $S(\PP^2,d(d+3)/6)$.
{Since} $3(x-1)<\binom{d+2}{2}$, {we have $h^0(\Ii_{2S'}(d))>0$.} 
Take  a general $C\in |\Ii_{2S'}(d)|$. 
Since $S'$ is general, {by the Alexander-Hirschowitz Theorem \ref{AHthm},} we have $h^0(\Ii_{2S'}(d)) =1$.
{Moreover, $H^0(\Ii_{2S}(d))=H^0(\Ii_{2S'}(d))$, since  $h^0(\Ii_{2S}(d))>0$.}
Hence 
{$|\Ii_{2S'}(d)|=\{C\}=|\Ii_{2S}(d)|$}, that is
$C$ is unique, once $S'$ is fixed. 
By {\cite[{Theorem 1.2}]{cov1}} $C$ has no multiple components,
hence 
$\mathrm{Sing}(C)$ is finite 
and $S\subseteq \mathrm{Sing}(C)$.

It follows that $$\dim(K)=2\#(S')+{\dim}|\Ii_{2S'}(d)|=2x-2,$$ which is a contradiction. 
\end{proof}

\begin{example}\label{ex2} 
Take $n=2$, $d=6$ and $x=10$. We find an irreducible component $K\subset \TT(2,6;10)$ of dimension $19$ and with $\eta(K)=9$, the maximal possible spread for a $19$-dimensional family. 
Fix a general $B\in S(\PP^2,9)$. Since $B$ is general, it is contained in a unique plane cubic, $C_B$.
The $19$-dimensional irreducible family $K$ is the closure of the set formed by all $B\cup \{p\}$, $B$ general in $S(\PP^2,9)$ and $p\in C_B\setminus B$. 
The Alexander-Hirschowitz Theorem \ref{AHthm} gives $\dim \TT(2,6;10)\le 19$. Thus $K$ is an irreducible component of $\TT(2,6;10)$.\end{example} 
\begin{example}\label{ex-nuovo}
Analogously to the previous example, it is easy to show that there is a component $K$ of $\TT(2,8;15)$ of dimension $29$.
Indeed $14$ general points, $B\in S(\PP^2,14)$, are contained in a unique plane quartic, $C_B$. Let $K$ be the closure of the set of $A=B\cup \{p\}$, where $B\in S(\PP^2,14)$ is general and $p\in C_B\setminus B$. 
Note that $\eta(K)=14$.
\end{example}

\begin{lemma}\label{Lem} 
Fix integers $d$ and $x$ such that $d\ge 5$ and $x <(d+2)(d+1)/6$. 
Then $\phi(V_{x,d})$ is an open dense set of $S(\PP^2,x)$. 
\end{lemma} 
\begin{proof}
Take a general $A\in S(\PP^2,x)$. By the Alexander-Hirschowitz Theorem \ref{AHthm} 
we have $\dim |\Ii_{2A}(d)| = (d^2+3d)/2-3x$. Recall that $\dim(V_{x,d}) =(d^2+3d)/2 -x$. 
By \cite[Theorem 5.1]{cc2} 
a general $T_A\in |\Ii_{2A}(d)|$ is nodal 
and {the points of} $A$ are the only singular points of $T_A$. 
By varying $A\in S(\PP^2,x)$ we see that for a general $C\in V_{x,d}$ 
the set $\mathrm{Sing}(C)$ is a general element of $S(\PP^2,x)$. 
\end{proof}

\begin{notation}
For any $d$ and $g$ such that $d>0$ and $1-d\le g\le (d-1)(d-2)/2$, let $V(d,g)$ be the {closure of the }set of all degree $d$ curves $C\subset \PP^2$ without multiple components and of geometric genus $g$.
Let $V(d,g)^{\mathrm{irr}}$ denote the set of all irreducible $C\in V(d,g)$. 
\end{notation}

\begin{remark}\label{rmk-severi-open}
Of course, $V(d,g)^{\mathrm{irr}}=\emptyset$ if $g<0$. Harris proved that for all $0\le g\le (d-2)(d-1)/2$ 
the set $V(d,g)^{\mathrm{irr}}$ is irreducible \cite[$(\ast)$ in the Introduction]{harris}, 
hence {the Severi variety} $V_{(d-1)(d-2)/2-g,d}$ {(see Notation \ref{severi})} is an open dense subset of it. 
\end{remark}

The crucial result in this section is the following: 
\begin{proposition}\label{qfin} 
Fix a positive integer $d\ge 7$ and let
%
\begin{eqnarray*}
x &=&  \frac{1}{3}\binom{d+2}{2}-1=\frac{(d+4)(d-1)}{6}, \quad\mbox{ if }d\equiv 1,2\pmod{3};\\
x &=& \left\lfloor \frac{1}{3}\binom{d+2}{2}\right\rfloor=\frac{d(d+3)}{6}, \quad\mbox{  if }d\equiv 0\pmod{3}.
\end{eqnarray*}
Then we have $\dim\TT(2,d;x)\le 2x-2$.


\end{proposition}

\begin{proof} 
Assume by contradiction that there exists a component   $K$ of $\TT(2,d;x)$ of dimension $2x-1$.
By Remark \ref{rmk:eta}, we have $\eta(K)=x-1$.

For any $S\in K$, by definition of Terracini locus, we have 
${\epsilon:=}h^1(\Ii_{2S}(d))>0$ hence 
\begin{equation}\label{formula-ini}h^0(\Ii_{2S}(d))= \binom{d+2}2-3x+\epsilon\ge2+\frac{d^2+3d}2-3x,
\end{equation}
which implies $\dim |\Ii_{2S}(d)| \ge 1+\frac{d^2+3d}2-3x$. 

Consider the set 
$$\Gamma_K=\{(S,C):S\in K, C\in|\Ii_{2S}(d)|\}\subset S(\PP^2,x)\times |\Oo_{\PP^2}(d)|$$
 and let  
$A_K\subseteq \Gamma_K$ be a {dominant irreducible component of $\pi_1^{-1}(K)$}. 
Note that $A_K$ is non-empty because any element of $K$ is in the Terracini locus.

Let ${\pi_2}: A_K\to |\Oo_{\PP^2}(d)|$ denote the restriction to $A_K$ 
of the projection to the second factor. Since each fiber of ${\pi_1}$ has dimension {greater than or equal to}
$1+(d^2+3d)/2 -3x$ and $\dim K=2x-1$, we have 
\begin{equation}\label{dimAK}\dim A_K\ge \frac{d^2+3d}{2}-x =\dim V_{x,d}.\end{equation}

{Take a general} $(S,C)\in A_K$. Since $\eta(K)=x-1$, then there is a general $S'\in S(\PP^2, x-1)$ {contained in $S$.}
Fix such $S'$.

{In the sequel of the proof (steps (b) and (c)), we will use the following two claims.}

\medskip

\quad {\bf Claim 1:} By the Alexander-Hirschowitz Theorem \ref{AHthm}, we have $h^0(\Ii_{2S'}(d-1)) =0$.
{In particular $\deg(T) \ge d$ for all curves $T\supset 2S'$.}

\medskip
Claim 1 implies the following:
\medskip

\quad {\bf Claim 2:} 
{Every irreducible component $D$ of $C$ contains at least one point of $S'$.}

\medskip

We consider now the following three cases: $C$ irreducible and reduced, $C$ reducible and reduced, and $C$ non-reduced.

\medskip

\quad {\bf (a)} Assume first that $C$ is irreducible and reduced.
Since $C$ has no multiple components, then $\mathrm{Sing}(C)$ is finite. Thus a general fiber of ${\pi_2}$ is finite, hence $\dim {\pi_2}(A_K) =\dim A_K$. 
Therefore by \eqref{dimAK} we have $\dim \pi_2(A_K)\ge \dim V_{x,d}$. 
Hence ${\pi_2}(A_K)$ contains a non-empty open subset of $V_{x,d}$, {by Remark \ref{rmk-severi-open}}.
{Since $x <(d+2)(d+1)/6$,} by Lemma \ref{Lem}, 
we get $\dim\TT(2,d;x)=2x$, which is false. 

\medskip

\quad {\bf (b)} 
Assume now $C$ reduced and reducible. 
Since $C$ has no multiple components, we have $$\dim {\pi_2}(A_K) =\dim A_K{=\binom{d+2}{2} -x-1 +\epsilon},$$ where 
$\epsilon:= h^1(\Ii_{2S}(d)) >0$.

Write $C = C_1\cup \cdots \cup C_s$, $s\ge 2$,
with $C_i$ an {irreducible} curve of degree $d_i$ and geometric genus $g_i$, {for all $1\le i\le s$}.
{
Restricting to a non-empty Zariski open subset $U_K$ of $K$, we may assume that the integers $s$, $d_i$ and $g_i$ are the same for all $C\in U_K$.} We have $d_1+\cdots +d_s=d$.

By \cite[{Introduction}]{harris}, the family 
{$\Aa_i$}
of all 
irreducible plane curves of degree $d_i$ and geometric genus $g_i$ is irreducible of dimension 
$$\dim \Aa_i =(d_i^2+3d_i)/2 -(d_i-1)(d_i-2)/2 +g_i = 3d_i -1+g_i.$$ 

{Obviously, $\dim \pi_2(A_K)$ is at most 
$\sum_{i=1}^{s} \dim \Aa_i=\sum_{i=1}^{s} (3d_i-1+g_i) = 3d-s+g_1+\cdots +g_s$.}

 By \cite[{Theorem 1.2}]{nob}, 
we have $$g_1+\cdots +g_s \le \frac{(d-1)(d-2)}{2} -x+1,$$
{hence, we get
$$\dim \pi_2(A_K)=\binom{d+2}{2} -x-1+\epsilon \le 3d-s+\frac{(d-1)(d-2)}{2}-x+1
$$
from which we have
$s\le 2-\epsilon\le 1$, which contradicts our assumption.}

\medskip

\quad {\bf (c)} Now assume that $C$ has at least one multiple component. 

{
First, note that each irreducible component of $C_{\red}$ contains at least one point
of $S'$, by Claim 2. Moreover,
by Claims 1 and 2,
each irreducible component of $C$ has multiplicity at most $2$.}

{Hence, we can write $C =A_1+2A_2$, 
where $A_i$, for $1\le i\le 2$, is the union of the irreducible components of $C_{\red}$ appearing with multiplicity exactly $i$ in $C$.
Set $a_i:= \deg(A_i)$, so that $a_1+2a_2=d$, and
$B_2:= S'\cap A_2$,  $B_1:= S'\setminus B_2$ and $b_i:= \#B_i$. 
Note that $B_1\cap B_2=\emptyset$ and $b_1+b_2=x-1$.}

\medskip

{
\quad {\bf Claim 3:} 
Note that $a_1\neq 1$. Indeed if $a_1=1$, then we would have $b_1=0$ since the line $A_1$ cannot have singular points. Hence $2A_2$ would be in $|\Ii_{S'}(d-1)|$ which is empty by Claim 1.
Analogously, one can see that if $a_1=2$, then $b_1=1$.}

\medskip 

{Since $B_2\subset A_2$ and $B_2\subseteq S'$ is general in $S(\PP^2,b_2)$, we have
\begin{equation}\label{stima b2}
h^0(\Ii_{B_2}(a_2))=\binom{a_2+2}{2} - b_2>0.\end{equation}
}

{Now, let $\Bb$ denote the base locus of $|\Ii_S(d)|$.  
Bertini's Theorem implies that $C$ is smooth outside $\Bb$, hence $A_2\subseteq \Bb $ and 
 the residual exact sequence of $2A_2$ 
gives
$$h^0(\Ii_{2S}(d)) = h^0(\Ii_{2B_1}(d-2a_2)).$$
Recall from \eqref{formula-ini}, that $h^0(\Ii_{2S}(d)) =\binom{d+2}{2}-3x+\epsilon$, with $\epsilon:=h^1(\Ii_{2S}(d)) >0$. 
Recall also that $d-2a_2=a_1$.}

{
Since $B_1\subseteq S'$ is general in $S(\PP^2,b_1)$, by the Alexander-Hirschowitz Theorem \ref{AHthm} we have
$h^0(\Ii_{2B_1}(a_1))=\binom{a_1+2}{2}-3b_1,$ except if either $a_1=2$ and $b_1=2$, or $a_1=4$ and $b_1=5$.}

\medskip

{By Claim 3, the case $(a_1,b_1)=(2,2)$ is impossible. So assume first that $(a_1,b_1)\neq(4,5)$.}
{
Thus, we have
$$\binom{d+2}{2}-3x+\epsilon = \binom{a_1+2}{2}-3b_1
$$
which gives, since $x-1=b_1+b_2$,
$$\epsilon=3b_2+3-a_2(2a_1+2a_2+3),$$
}
{
and using \eqref{stima b2} we have
$$1\le \epsilon<3\binom{a_2+2}{2}+3-a_2(2a_1+2a_2+3)=\frac12\left(12-a_2(4a_1+a_2-3)\right).$$}

{Now, if $\frac12\left(12-a_2(4a_1+a_2-3)\right)\le 1$, we would have a contradiction. 
Hence we get
$\left(12-a_2(4a_1+a_2-3)\right)\ge 4$, that is
\begin{equation}\label{condizione}
a_2(4a_1+a_2-3)\le8.
\end{equation}
}

{Now, {we have} that the only pair $(a_1,a_2)$ satisfying  \eqref{condizione} and such that $d=2a_2+a_1\ge 7$, {$a_2>0$}, $a_1\neq1$ (by Claim 3) is $(a_1,a_2)=(0,4)$. 
Indeed if $a_1=0$, then \eqref{condizione} implies $a_2=4$.
If $a_1=2$ (hence $d$ is even and $d\ge 8$), we get $a_2=(d-2)/2\ge3$, contradicting \eqref{condizione}. If $a_1\ge 3$, since $a_2\ge 1$, \eqref{condizione} fails.

In this case we have
 $d=8$, $x=14$, $S$ is contained in a double plane quartic $C=2A_2$. We have $h^1(\Ii_{2S}(8)) >0$ if and only if $h^0(\Ii_{2S}(8)) \ge 4$. 
We have $h^1(\Ii_{2S}(8)) >0$ if and only if $h^0(\Ii_{2S}(8)) \ge 4$.
Since $2A_2$ is a general element of $|\Ii_{2S}(8)|$, then
$h^0(\Ii_S(4))=h^0(\Ii_{2S}(8)) \ge 4$. 

{Since $\eta(K)=x-1$, then there is a general $S'\in S(\PP^2,13)$ which is contained in $S$.} Hence $h^0(\Ii_{S'}(4))\ge h^0(\Ii_{S}(4)) \ge 4$. On the other hand, since $S'$ is general in $S(\PP^2,13)$, we have $h^0(\Ii_{S'}(4)) =2$, a contradiction.}

Finally, if $(a_1,b_1)=(4,5)$, we can repeat verbatim the same argument of the general case and we get again a contradiction. This completes the proof that $\dim(\TT(2,d;x))\le 2x-2$.
\end{proof}

{
\begin{lemma}\label{uniqueness}
Fix an integer $d\ge 7$ such that $d\equiv 1,2\pmod{3}$ and set $x:= (d+2)(d+1)/6$. Then the irreducible component described in Example \ref{e1} is the unique irreducible codimension one component of $\TT(2,d;x)$, except if $d=8$.
\end{lemma}}
\begin{proof}
{Take an irreducible component $T$ of $\TT(2,d;x)$ dimension $2x-1$ and take a general $S\in T$. By assumption $h^0(\Ii_{2S}(d)) \ne 0$. Take a general $D\in |\Ii_{2S}(d)|$.
Since $\eta(T)=x-1$ and $S$ is general in $T$ each $S'\subset S$ with $\#S'=x-1$ has the property that both $S'$ and $2S'$ have the Hilbert function of a general element.
If $D$ is irreducible and nodal, then $T$ is the component $\overline{\phi(V_{x,d})}$,  described in Example \ref{e1}.}

{Now assume that $D$ is irreducible and not nodal, hence $D$ has geometric genus strictly less than $(d-1)(d-2)/2 -x$. By \cite[{Introduction}]{harris} $\dim T < V_{x,d}$, a contradiction. The case $C$ reduced and reducible or with a multiple components are excluded as in the proof of {Proposition \ref{qfin}}. The only difference is that the exception
$(a_1,a_2)=(0,4)$ gives in this case $d=8$ and $x=15$ which is described in Example \ref{ex-nuovo}.}
\end{proof}

We are finally ready to prove our main result.

\smallskip

\noindent
\textit{Proof of Theorem \ref{finale2}.} 
The existence of an irreducible component of $\TT(2,d;x)$ of dimension $2x-1$ is proved:
\begin{itemize}
\item in
Example \ref{ex1} for $(d,x)=(4,4)$, 
\item in Proposition \ref{bo1} for $(d,x)=(4,6)$,  
\item in Remark \ref{casipiccoli}
for $(d,x)=(5,6),(6,9)$, 
\item in Example \ref{caso(5,7)}
for $(d,x)=(5,7)$ , 
\item in
Example \ref{ex2} 
for $(d,x)=(6,10)$,
\item
in Example \ref{e1} for 
$d\equiv 1,2\pmod{3}$, $d\ge7$ and $x = (d+2)(d+1)/6$.
\end{itemize}

Let $\rho=\rho(2,d)$ as in Notation \ref{notation rho}.
The other implication follows 
\begin{itemize}
\item for $d=4$ and $x\ge 7$ from Remark \ref{rem-deg-4},
(recall that the case $(d,x)=(4,5)$ is known by the Alexander-Hirschowitz Theorem \ref{AHthm}), 
\item
for $d\ge5$ and $x\ge \rho+1$, from  Theorem \ref{000u2},  
\item
for $d\equiv 0\pmod{3}$, 
$d\ge 9$ and
$x = \rho$ from
Lemma \ref{last}, 
\item {for $d=5$ and $x=5$, from Remark \ref{casipiccoli}},
\item {for $d=6$ and $x=8$, from Proposition \ref{caso 6,8} and for $d=6$ and $5\le x\le 7$, from Proposition \ref{prop sigma}},
\item for $d\ge 7$ and $x<\rho$, from Proposition \ref{qfin}.
\end{itemize}

In the case $(d,x)=(8,15)$ one component is described in Example \ref{ex-nuovo}, the other component is given by $\overline{\phi(V_{15,8})}$, see Example \ref{e1}.
{The uniqueness in the other cases is proved in Lemma \ref{uniqueness}.}
\qed

\end{document}